\documentclass{amsart}

\RequirePackage{fix-cm}
\RequirePackage[l2tabu, orthodox]{nag}
\RequirePackage[all, warning]{onlyamsmath}

\usepackage{latexsym}
\usepackage{amsmath,amsfonts,amssymb}
\usepackage{amsthm}
\usepackage{graphicx,color,xcolor}
\usepackage{overpic}
\usepackage{subcaption}
\usepackage{braket}
\usepackage{mathrsfs}
\usepackage[bookmarks=true,bookmarksnumbered=true]{hyperref}
\usepackage{multirow,bigdelim}
\usepackage{enumerate}
\usepackage{cleveref}
\usepackage{comment}
\usepackage{breqn}
\numberwithin{equation}{section}
\usepackage{doi}\usepackage{hyperref}

\newtheorem{theorem}{Theorem}[section]
\newtheorem{lemma}[theorem]{Lemma}
\newtheorem{corollary}[theorem]{Corollary}
\newtheorem{proposition}[theorem]{Proposition}

\theoremstyle{definition}

\newtheorem{remark}[theorem]{Remark}
\crefname{section}{Section}{Sections}
\crefname{appendix}{Appendix}{Appendices}
\crefname{theorem}{Theorem}{Theorems}
\crefname{lemma}{Lemma}{Lemmas}
\crefname{corollary}{Corollary}{Corollaries}			
\crefname{proposition}{Proposition}{Propositions}	
\crefname{claim}{Claim}{Claims}
\crefname{conjecture}{Conjecture}{Conjectures}			
\crefname{definition}{Definition}{Definitions}
\crefname{problem}{Problem}{Problems}
\crefname{example}{Example}{Examples}
\crefname{remark}{Remark}{Remarks}
\crefname{figure}{Figure}{Figures}
\crefname{footnote}{Footnote}{Footnotes}
\crefname{equation}{}{}
\crefname{enumi}{}{}
\crefformat{enumi}{(#2#1#3)}
\Crefformat{enumi}{(#2#1#3)}

\newcommand{\QED}{\hfill \ensuremath{\Box}}

\newcommand{\R}{\mathbb{R}}
\newcommand{\Q}{\mathbb{Q}}

\newcommand{\N}{\mathbb{N}}

\newcommand{\ld}{,\ldots,}

\newcommand{\ep}{\varepsilon}

\newcommand{\norm}[1]{\left\|#1\right\|}

\newcommand{\D}{\displaystyle}

\newfont{\bg}{cmr9 scaled\magstep2}
\newcommand{\bigzerol}{\smash{\lower1.0ex\hbox{\bg 0}}}
\newcommand{\al}{\alpha}
\newcommand{\Hess}{\mathrm{Hess}}
\DeclareMathOperator{\rank}{rank}

\DeclareMathOperator{\supp}{supp}

\title[
Non-density of $C^0$-stable mappings on non-compact manifolds
]
{
Non-density of $C^0$-stable mappings 
\\
on non-compact manifolds
}

\author{Shunsuke Ichiki
}
\address{
Department of Mathematical and Computing Science,
School of Computing,
Tokyo Institute of Technology,
Tokyo 152-8552,
Japan}
\email{ichiki@c.titech.ac.jp}

\begin{document}
\date{}
\begin{abstract}
The problem of density of $C^0$-stable mappings is a classical and venerable subject in singularity theory.
In 1973, Mather showed that the set of proper $C^0$-stable mappings is dense in the set of all proper mappings, which implies that the set of $C^0$-stable mappings is dense in the set of all mappings if the source manifold is compact.
The aim of this paper is to complement Mather's result and to provide new information to the subject.
Namely, we show that the set of $C^0$-stable mappings is never dense in the set of all mappings if the source manifold is non-compact.
As a corollary of this result and Mather's result, we can obtain a characterization of density of $C^0$-stable mappings, i.e., the set of $C^0$-stable mappings is dense in the set of all mappings if and only if the source manifold is compact.
To prove the non-density result, we provide a more essential result by using the notion of topologically critical points. 
\end{abstract}
\maketitle
\noindent

\section{Introduction}\label{sec:intro}
Let $N$ and $P$ be $C^\infty$-manifolds.
Two $C^\infty$-mappings $f$ and $g$ from $N$ into $P$ are said to be $C^\infty$-equivalent (resp. $C^0$-equivalent) if there exist $C^\infty$-diffeomorphisms (resp. homeomorphisms) $\Phi:N\to N$ and $\Psi :P\to P$ such that $g=\Psi \circ f \circ \Phi^{-1}$.  
We denote the space of all $C^\infty$-mappings of $N$ into $P$ (resp. the space of all proper $C^\infty$-mappings of $N$ into $P$) endowed with the Whitney $C^\infty$-topology by $C^\infty(N,P)$ (resp. $C^\infty_{pr}(N,P)$).
We say that $f$ is \emph{$C^\infty$-stable} (resp. \emph{$C^0$-stable}) if there exists an open neighborhood $\mathscr{U}$ of $f$ such that any mapping in $\mathscr{U}$ is $C^\infty$-equivalent (resp. $C^0$-equivalent) to $f$. 
In what follows, unless otherwise stated, all manifolds and mappings are of class $C^\infty$.

Before describing $C^0$-stability, which is the main topic of this paper, we first review the previous studies on density of $C^\infty$-stable mappings.
On the $C^\infty$-stability, in a celebrated series around 1970  \cite{Mather1968,Mather1968b,Mather1969,Mather1969b,Mather1970,Mather1971}, Mather established a significant theory and he gave a characterization of density of proper $C^\infty$-stable mappings in $C^\infty_{pr}(N,P)$ as follows:
\begin{theorem}[\cite{Mather1971}]\label{thm:mather}
Let $N$ and $P$ be manifolds of dimensions $n$ and $p$, respectively.
Then, the set of all proper $C^\infty$-stable mappings is dense in $C^\infty_{pr}(N,P)$ if and only if the pair $(n,p)$ satisfies one of the following conditions.
\begin{enumerate}[$(1)$]
\item 
$n<\frac{6}{7}p+\frac{8}{7}$ and $p-n\geq 4$
\item 
$n<\frac{6}{7}p+\frac{9}{7}$ and $3\geq p-n\geq 0$
\item 
$p<8$ and $p-n=-1$
\item 
$p<6$ and $p-n=-2$
\item 
$p<7$ and $p-n\leq-3$
\end{enumerate}
\end{theorem}
A dimension pair $(n,p)$ is called \emph{nice} if it satisfies one of the conditions (1)--(5) in \cref{thm:mather}.
Since $C^\infty_{pr}(N,P)=C^\infty(N,P)$ in the case where $N$ is compact, \cref{thm:mather} yields a characterization of density of $C^\infty$-stable mappings in $C^\infty(N,P)$ in the case.

After that, the case where a source manifold is non-compact was considered by Dimca, and in 1979,
he gave the following result.
\begin{proposition}[\cite{Dimca1979}]\label{thm:dimca}
Let $N$ be a non-compact manifold.
Then, the set of all $C^\infty$-stable mappings is not dense in $C^\infty(N,\R)$.
\end{proposition}
Then, in \cite{Ichiki2022}, for a non-compact source manifold $N$ and an arbitrary target manifold $P$, the following has been shown rigorously:
\begin{theorem}[\cite{Ichiki2022}]\label{thm:main2}
Let $N$ be a non-compact manifold, and $P$ a manifold.
Then, the set of all $C^\infty$-stable mappings is never dense in $C^\infty(N,P)$.
\end{theorem}
By combining Mather's theorem (\cref{thm:mather}) and \cref{thm:main2}, we obtain the following characterization of density of $C^\infty$-stable mappings in $C^\infty(N,P)$ in the case where $N$ is not necessarily compact.
\begin{corollary}[\cite{Ichiki2022,Mather1971}]\label{thm:combine}
Let $N$ and $P$ be manifolds of dimensions $n$ and $p$, respectively.
Then, the set of all $C^\infty$-stable mappings is dense in $C^\infty(N,P)$ if and only if $N$ is compact and $(n,p)$ is nice.
\end{corollary}
The problem of density of $C^0$-stable mappings is also a classical and venerable subject in singularity theory.
On the $C^0$-stability, Mather established the following theorem in 1973, which is so significant as well as \cref{thm:mather}, and the bibles \cite{Wall1995,Gibson1976,Ruas2022} have been published.
\begin{theorem}[\cite{Mather1973a}]\label{thm:mather_t}
Let $N$ and $P$ be manifolds.
Then, the set of all proper $C^0$-stable mappings is dense in $C^\infty_{pr}(N,P)$.
\end{theorem}

\cref{thm:mather_t} implies that the set of all $C^0$-stable mappings is always dense in $C^\infty(N,P)$ if $N$ is compact.
The aim of this paper is to complement Mather's significant theorem (\cref{thm:mather_t}) and to provide new information to the classical and venerable subject in singularity theory, by considering the case where $N$ is non-compact as follows: 
\begin{theorem}\label{thm:main}
Let $N$ be a non-compact manifold, and $P$ a manifold.
Then, the set of all $C^0$-stable mappings is never dense in $C^\infty(N,P)$.
\end{theorem}
Furthermore, by combining \cref{thm:mather_t} and \cref{thm:main}, we can obtain the following characterization of density of $C^0$-stable mappings.
\begin{corollary}
Let $N$ and $P$ be manifolds.
Then, the set of all $C^0$-stable mappings is dense in $C^\infty(N,P)$ if and only if $N$ is compact.
\end{corollary}
\begin{remark}
We give the following remarks on \cref{thm:main}.
\begin{enumerate}
    \item 
    In \cref{sec:preparation}, we state a more essential result (see \cref{thm:main_a}) than \cref{thm:main} by using the notion of topologically critical points, which is the main theorem of this paper, and \cref{thm:main} follows from the result as a corollary.
    \item
    \cref{thm:main} implies \cref{thm:main2} since any $C^\infty$-stable mapping is $C^0$-stable.
\end{enumerate}
\end{remark}
The remainder of this paper is organized as follows.
In \cref{sec:preparation}, we state the main theorem, which is a more essential result than \cref{thm:main}, and preparations for the proof.
In \cref{sec:proof}, we show the main theorem.

\section{Main theorem and preparations for the proof}\label{sec:preparation}
Let $N$ and $P$ be manifolds, and $f : N\to P$ a mapping.
A point $q\in N$ is called a \emph{topologically critical point} (resp. \emph{critical point}) of $f$ if $f|_U:U\to P$ is not an open mapping for any open neighborhood $U$ of $q$ (resp. $\rank df_q<\dim P$).
\begin{remark}\label{rem:top}
We give the following remarks on topologically critical points.
\begin{enumerate}
\item
    If $q\in N$ is a topologically critical point of $f:N\to P$, then $q$ is a critical point as follows:
    Suppose that $q\in N$ is not a critical point.
    Then, since $\rank df_q\geq \dim P$, there exists an open neighborhood $U$ such that $f|_U:U\to P$ is an open mapping by the implicit function theorem, which contradicts the hypothesis that $q$ is a topologically critical point.
    \item
    A critical point is not necessarily a topologically critical point.
    For example, let $f:\R\to \R$ be the function defined by $f(x)=x^3$.
    Then, $x=0$ is a critical point of $f$, however the point is not a topologically critical point.
     \item \label{rem:top_p}
     Suppose that $f,g\in C^\infty(N,P)$ are $C^0$-equivalent, that is, $g=\Psi \circ f \circ \Phi^{-1}$, where $\Phi:N\to N$ and $\Psi:P\to P$ are homeomorphisms.
     If $q\in N$ is a topologically critical point of $f$, then $\Phi(q)$ is a topologically critical point of $g$.
     Namely, topologically critical points can be preserved by homeomorphisms although usual critical points are not necessarily preserved by them.
    \end{enumerate}
\end{remark}
The following is the main theorem, and \cref{thm:main} is an easy consequence of this theorem.
\begin{theorem}\label{thm:main_a}
    Let $N$ be a non-compact manifold, and $P$ a manifold.
    Then, there exists a non-empty open subset $\mathscr{U}$ of $C^\infty(N,P)$ such that for any positive integer $d$ satisfying $d\geq 2$,
    \begin{align*}
        \mathscr{U}_d=\set{f\in \mathscr{U}|\mbox{$f$ has $d$ topologically critical points with the same image}}
    \end{align*}
    is dense in $\mathscr{U}$.
    In particular, any mapping in $\mathscr{U}$ is not $C^0$-stable.
\end{theorem}
In \cref{thm:main_a}, when we show that any mapping in $\mathscr{U}$ is not $C^0$-stable, we use the simple fact that topologically critical points are preserved by homeomorphisms as mentioned in \cref{rem:top}\cref{rem:top_p}, which is an advantage of topologically critical points compared to usual critical points.

In what follows, for a given positive integer $m$, we denote the origin $(0\ld 0)$ of $\R^m$ by $0$, the Euclidean norm of $x\in \R^m$ by $\|x\|$, and the $m$-dimensional open ball with center $x\in \R^m$ and radius $r>0$ by $B^m(x,r)$, that is,
\begin{align*}
B^m(x,r)=\set{x'\in \R^m|\|x-x'\|<r}.
\end{align*}
For a set (resp. a topological space) $X$ and a subset $A$ of $X$, we denote the complement of $A$ (resp. the closure of $A$) by $A^c$ (resp. $\overline{A}$).

The former assertion on a usual critical point of the following lemma is the same as \cite[Lemma~2.1]{Ichiki2022}, and we update the lemma by adding the latter assertion on topologically critical points as follows:
\begin{lemma}\label{thm:critical}
Let $f=(f_1\ld f_p):B^n(0,r)\to \R^p$ $(r>0)$ be a mapping such that 
\begin{align*}
    f_p(x)=\frac{1}{2}\sum_{i=1}^nx_i^2+a,
\end{align*} where $a$ is a real number and $x=(x_1\ld x_n)$.
If $g=(g_1\ld g_p):B^n(0,r)\to \R^p$ satisfies
\begin{align}\label{eq:d0}
\sqrt{\sum_{i=1}^n\left(\frac{\partial f_p}{\partial x_i}(x)-
    \frac{\partial g_p}{\partial x_i}(x)\right)^2}
    <\frac{r}{2}
\end{align}
for any $x\in B^n(0,r)$, then there exists a critical point $x_0$ of $g$ in $B^n(0,r)$.
Moreover, if $\Hess(g_p)_{x_0}$ is positive definite, then $x_0$ is a topologically critical point of $g$, where $\Hess(g_p)_{x_0}$ is the Hessian matrix of $g_p$ at $x_0$.
\end{lemma}
\begin{proof}[Proof of \cref{thm:critical}]
By the same proof as \cite[Lemma~2.1]{Ichiki2022}, it follows that there exists a critical point $x_0$ of $g$ in $B^n(0,r)$ such that $\frac{\partial g_p}{\partial x_i}(x_0)=0$ for any $i\in \set{1\ld n}$.

Suppose that $\Hess(g_p)_{x_0}$ is positive definite.
Since $x_0$ is also a critical point of $g_p$, there exists a coordinate neighborhood $(U,\varphi)$ of $B^n(0,r)$ containing $x_0$ ($\varphi(x_0)=0\in \R^n$) such that $g_p\circ \varphi^{-1}:\varphi(U)\to \R$ has the following form:
\begin{align*}
        (g_p\circ \varphi^{-1})(t_1\ld t_n)=g_p(x_0)+t_1^2+\cdots +t_n^2,
\end{align*}
which implies that $x_0$ is a topologically critical point of $g$.
\end{proof}
\section{Proof of the main theorem}\label{sec:proof}
The proof is divided into the following three steps:
In STEP~1, we construct the $C^\infty$ mapping $f:N\to P$ defined by \eqref{eq:f}.
In STEP~2, we construct the open neighborhood $\mathscr{U}$ of $f$ in $C^\infty(N,P)$ given by \eqref{eq:u}, and we provide a lemma on properties of a mapping in $\mathscr{U}$ (see \cref{thm:u}).
Finally, in STEP~3, after preparing two lemmas (\cref{thm:contain,thm:summary}), we show that $\mathscr{U}_d$ is dense in $\mathscr{U}$, and by using this assertion and the simple fact that topologically critical points are preserved by homeomorphisms, we prove that any mapping in $\mathscr{U}$ is not $C^0$-stable.

The method of the proof is almost the same as that of \cref{thm:main2}, but the most important difference is the addition of condition (c) in the definition of $O_\al$ in \cref{eq:o} of STEP~2 to deal with topological critical points instead of ordinary critical points (more precisely, to use \cref{thm:critical}).
In fact, in the proof of \cref{thm:main2}, there are only conditions (a) and (b) in the definition of $O_\al$, and the open set $\mathscr{U}$ in \cref{eq:u} is defined by using an open subset of the 1-jet space $J^1(N,P)$.
On the other hand, in this proof, since we add condition (c), we define $\mathscr{U}$ in \cref{eq:u} using an open set in the 2-jet space $J^2(N,P)$.

STEP 1 is the same as that in \cref{thm:main2}, but is also described in this paper for the sake of readers' convenience since in this first step we introduce some symbols that will be subsequently used.

\smallskip 
\underline{STEP~1}.
Set $n=\dim N$, $p=\dim P$ and $\ell=2n+1$.
By Whitney's embedding theorem, there exist an embedding $F:N\to \R^\ell$ such that $F(N)$ is a closed subset of $\R^\ell$.
Then, there exists a point $z_0\in \R^\ell\setminus F(N)$.
Since $N$ is non-compact, $F(N)$ is also non-compact.
Thus, $F(N)$ is not bounded, which implies that there exists a sequence $\set{R_\al}_{\al\in \N}$ of positive real numbers and a sequence $\set{z_\al}_{\al\in \N}$ of points in $\R^\ell$ such that 
\begin{itemize}
    \item 
    $R_\al<R_{\al+1}$ for any $\al\in \N$ and $\D\lim_{\al \to \infty}R_\al=\infty$,
    \item
    $z_\al\in F(N)\cap (B^\ell(z_0,R_{\al+1})\setminus \overline{B^\ell(z_0,R_{\al})})$ for any $\al\in \N$.
\end{itemize}
Let $\al$ be a positive integer.
Set $q_\al=F^{-1}(z_\al)$.
Here, note that 
\begin{align*}
    F^{-1}(B^\ell(z_0,R_{\al+1})\setminus \overline{B^\ell(z_0,R_{\al})})
\end{align*}
    is an open neighborhood of $q_\al$.
Then, there exists a coordinate neighborhood $(U_\al,\varphi_\al)$ of $N$ with the following properties:
\begin{itemize}
    \item 
    $\overline{U_\al}$ is compact,
    \item
    $q_\al\in U_\al\subset  F^{-1}(B^\ell(z_0,R_{\al+1})\setminus \overline{B^\ell(z_0,R_{\al})})$,
    \item
    $\varphi_\al(q_\al)=0\in \R^n$.
\end{itemize}
Moreover, there exist an open neighborhood $U_\al'$ of $q_\al$ and $\rho_\al:N\to \R$ ($0\leq \rho_\al(q)\leq 1$) such that 
\begin{itemize}
    \item 
    $\overline{U_\al'}\subset U_{\al}$,  
    \item
    $\rho_\al(q)=1$ for any $q\in \overline{U_\al'}$,
    \item
    $\supp \rho_\al\subset U_\al$,
\end{itemize}
where $\supp \rho_\al=\overline{\set{q\in N|\rho_\al(q)\not=0}}$.
Notice that $\supp \rho_\al$ is compact since $\overline{U_\al}$ is compact.
By choosing $U_\al'$ smaller for each $\al\in \N$ we can assume that 
\begin{itemize}
    \item 
    $\varphi_\al(U_\al')=B^n(0,r_\al)$,
    \item 
    $\D\lim_{\al \to \infty}r_\al=0$,
\end{itemize}
where each $r_\al$ is a positive real number.

Let $\gamma=(\gamma_1\ld \gamma_p):\N\to \Q^p$ be a bijection, and let $\eta_\al:\varphi_\al(U_\al)\to \R^p$ be the mapping defined by 
\begin{align*}
\eta_\al(x)=
\D\left(\gamma_1(\al)\ld \gamma_{p-1}(\al),\frac{1}{2}\sum_{i=1}^nx_i^2+\gamma_{p}(\al)\right)
\end{align*}
for each $\al\in \N$, where $x=(x_1\ld x_n)$.
Let $(V,\psi)$ be a coordinate neighborhood of $P$ that satisfies $\psi(V)=\R^p$.
Since $U_\al\cap U_\beta=\varnothing$ if $\al\not=\beta$, we can define $f:N\to P$ as follows:
\begin{align}\label{eq:f}
f(q)=\left\{ \begin{array}{ll}
\D\psi^{-1}(\rho_\al(q)(\eta_\al\circ \varphi_\al)(q)) & \mbox{if $q\in U_\al$}, 
\\
\\ 
\D\psi^{-1}(0) & \mbox{if $q\not\in \bigcup_{\al\in \N}U_\al$}.
\end{array} \right.
\end{align}
We show that $f$ is of class $C^\infty$.
Let $q\in N$ be any point.
If $q\in \bigcup_{\al\in \N}U_\al$, then by the definition of $f$ it is clear that $f$ is of class $C^\infty$ at $q$.
Thus, we consider the case $q\in (\bigcup_{\al\in \N}U_\al)^c$.
Since $\D\lim_{\al\to \infty}R_\al=\infty$, there exists $\beta \in\N$ such that $q\in F^{-1}(B^\ell(z_0,R_\beta))$.
For simplicity, set
\begin{align*}
A=F^{-1}(B^\ell(z_0,R_\beta))\cap \left(\bigcup_{\al\in \N}\supp \rho_\al\right)^c.
\end{align*}
Note that $q\in A$.
Since $R_\al<R_{\al+1}$ for any $\al\in \N$, we have
\begin{align*}
F^{-1}(B^\ell(z_0,R_\beta))\subset (\supp \rho_\al)^c
\end{align*}
for each $\al\in \N$ satisfying $\al>\beta$.
Thus, we obtain 
\begin{align*}
    A&=F^{-1}(B^\ell(z_0,R_\beta))\cap \left(\bigcap_{\al\in \N}(\supp \rho_\al)^c\right)
    \\
    &=F^{-1}(B^\ell(z_0,R_\beta))\cap \left(\bigcap_{\al\leq \beta}(\supp \rho_\al)^c\right),
\end{align*}
which implies that $A$ is an open subset of $N$.
Since $\rho_\al |_{A}$ is a constant function with a constant value $0$ for each $\al\in \N$, the mapping $f|_A$ is also constant.
Thus, $f$ is of class $C^\infty$ at $q$.

\smallskip 
\underline{STEP~2}.
In this step, we construct the open neighborhood $\mathscr{U}$ of $f$ in $C^\infty(N,P)$ given by \eqref{eq:u}, and we provide a lemma on properties of a mapping in $\mathscr{U}$.
Since $z_0\in \R^\ell\setminus F(N)$, we can define the following continuous function $\delta:N\to \R$:
\begin{align*}
    \delta(q)=\frac{1}{\norm{F(q)-z_0}}.
\end{align*}
Let $\pi:J^2(N,P)\to N\times P$ be the natural projection defined by $\pi(j^2g(q))=(q,g(q))$.
Then, for each $\al\in \N$, set 

\begin{align}\label{eq:o}
    O_\al=\set{j^2g(q)\in \pi^{-1}(U_\al\times V)|\mbox{$j^2g(q)$ satisfies \cref{eq:t}, \cref{eq:d} and \cref{eq:hess}}},
\end{align}where

\begin{enumerate}[(a)]
    \item
    $\norm{(\psi\circ f)(q)-(\psi\circ g)(q)}<\delta(q)$\label{eq:t},
    \\ 
    \item
    $\D\sqrt{\sum_{i=1}^n\left(\frac{\partial (\psi_p\circ f\circ \varphi_\al^{-1})}{\partial x_i}(\varphi_\al(q))-
    \frac{\partial (\psi_p\circ g\circ \varphi_\al^{-1})}{\partial x_i}(\varphi_\al(q))\right)^2}<\frac{r_\al}{2}$,
    \label{eq:d} 
    \\ 
    \item 
    \mbox{$\Hess(\psi_p\circ g\circ \varphi_\al^{-1})_{\varphi_\al(q)}$ is positive definite,}\label{eq:hess}\end{enumerate}
where $\psi_p$ is the $p$-th component of $\psi$.
From \cref{eq:t}, \cref{eq:d} and \cref{eq:hess}, it is not hard to see that $O_\al$ is an open subset of $J^2(N,P)$.

We prove that $\bigcap_{\al\in \N}(\overline{U_\al'})^c$ is an open subset of $N$.
Let $q\in \bigcap_{\al\in \N}(\overline{U_\al'})^c$ be any point.
Since $\D\lim_{\al\to \infty}R_\al=\infty$, there exists $\beta\in \N$ such that $q\in F^{-1}(B^\ell(z_0,R_\beta))$.
Since $R_\al<R_{\al+1}$ for each $\al\in \N$, we obtain $F^{-1}(B^\ell(z_0,R_\beta))\subset (\overline{U_\al'})^c$ for any $\al\in \N$ satisfying $\al>\beta$, which implies that 
\begin{align*}
    F^{-1}(B^\ell(z_0,R_\beta))\cap 
    \left(\bigcap_{\al\leq \beta}(\overline{U_\al'})^c\right)\subset \bigcap_{\al\in \N}(\overline{U_\al'})^c.
\end{align*}
Since the left-hand side of the above expression is an open neighborhood of $q$, it follows that $\bigcap_{\al\in \N}(\overline{U_\al'})^c$ is open.
Thus, since $\pi$ is continuous, 
\begin{align*}
    O:=\left(\bigcup_{\al\in \N}O_\al\right)\cup \pi^{-1}\left(\left(\bigcap_{\al\in \N}\left(\overline{U_\al'}\right)^c\right)\times V\right)
\end{align*}
is open in $J^2(N,P)$. 
Therefore, we can construct the following open subset of $C^\infty(N,P)$:
\begin{align}\label{eq:u}
    \mathscr{U}:=\set{g\in C^\infty(N,P)|j^2g(N)\subset O}.
\end{align}
By showing that $j^2f(N)\subset O$, we will prove that $\mathscr{U}\not=\varnothing$.
Let $j^2f(q)$ $(q\in N)$ be any element of $j^2f(N)$.
If there exists $\al \in \N$ such that $q\in U_\al$, then $j^2f(q)\in O_\al$ $(\subset O)$ since $f(q)\in V$ and $j^2f(q)$ clearly satisfies \cref{eq:t}, \cref{eq:d} and \cref{eq:hess}.
When $q\not\in \bigcup_{\al\in \N}U_\al$,  since \begin{align}\label{eq:n}
    N=\left(\bigcup_{\al\in \N}U_\al\right) 
    \cup
    \left(\bigcap_{\al\in \N}\left(\overline{U_\al'}\right)^c\right),
\end{align}
it must follow that $q\in \bigcap_{\al\in \N}\left(\overline{U_\al'}\right)^c$.
Therefore, since $f(q)\in V$, we obtain 
\begin{align*}
    j^2f(q)\in \pi^{-1}\left(\left(\bigcap_{\al\in \N}(\overline{U_\al'})^c\right)\times V\right)\ (\subset O),
\end{align*}
which implies that $\mathscr{U}\not=\varnothing$.

The following lemma describes properties of a mapping in $\mathscr{U}$ and it is an upgrade of \cite[Lemma~3.1]{Ichiki2022} to a claim about a topological critical point instead of a usual critical point.
\begin{lemma}\label{thm:u}
For any mapping $g\in \mathscr{U}$, we have $g(N)\subset V$ and there exists a sequence $\set{q_\al'}_{\al\in \N}$ of points in $N$ with the following properties.
\begin{enumerate}[$(1)$]
    \item \label{thm:u_critical}
    For each $\al \in \N$, $q_\al'$ is a topologically critical point of $g$ in $U_\al'$.
    \item \label{thm:u_dense}
    The set $\set{g(q_\al')|\al\in\N}$ is dense in $V$.
\end{enumerate}
\end{lemma}
\begin{proof}[Proof of \cref{thm:u}]
From the definition of $\mathscr{U}$, we have  $g(N)\subset V$.

Let $\al$ be any positive integer.
Then, we have 
\begin{align*}
    (\psi_p\circ f\circ \varphi_\al^{-1})(x)=\frac{1}{2}\sum_{i=1}^nx_i^2+\gamma_p(\al)
\end{align*}
for any $x=(x_1\ld x_n)\in \varphi_\al(U_\al')$ $(=B^n(0,r_\al))$.
For any $q\in U_\al'$, we obtain $j^2g(q)\in O_\al$ since we have \cref{eq:n} and $U_\al'$ is contained in $U_\al$ which does not intersect with $U_\beta$ $(\beta\not=\al)$.
Hence, $(\psi_p\circ g\circ \varphi_\al^{-1})|_{B^n(0,r_\al)}$ satisfies \cref{eq:d} and \cref{eq:hess}, which implies that there exists a topologically critical point of $(\psi_p\circ g\circ \varphi_\al^{-1})|_{B^n(0,r_\al)}$ in $B^n(0,r_\al)$ by \cref{thm:critical}.
Namely, there exists a topologically critical point of $g$ in $U_\al'$.
We denote this point by $q_\al'$.

Since $\set{q_\al'}_{\al\in \N}$ satisfies \cref{thm:u_critical}, it is sufficient to prove that the sequence of points also satisfies \cref{thm:u_dense}.
Let $V'$ be any open subset of $V$.
We show that $\set{g(q_\al')|\al\in \N}\cap V'\not=\varnothing$.
Then, by choosing $V'$ smaller, we can assume that $\psi(V')=B^p(y_0,\ep)$, where $y_0$ is a point of $\R^p$ and $\ep$ is a positive real number.
Note that for each $\al\in \N$, we have 
\begin{equation}
\begin{split}\label{eq:al_0}
    \norm{(\psi\circ g)(q_\al')-y_0}\leq 
    &\norm{(\psi\circ g)(q_\al')-(\psi\circ f)(q_\al')}
    +  \\
    &\norm{(\psi\circ f)(q_\al')-(\psi\circ f)(q_\al)}
  +
    \norm{(\psi\circ f)(q_\al)-y_0}.
   \end{split}
\end{equation}

Since 
\begin{align*}
    \delta(q_\al')=\frac{1}{\norm{F(q_\al')-z_0}}<\frac{1}{R_\al}
\end{align*}
for any $\al\in \N$ and $\D\lim_{\al\to \infty}R_\al=\infty$, there exists $\al_1\in \N$ such that $\delta(q_\al')<\frac{\ep}{3}$ for any $\al\in \N$ satisfying $\al\geq \al_1$.
Here, note that for each $\al\in \N$, we get 
\begin{align*}
        \norm{(\psi\circ g)(q_\al')-(\psi\circ f)(q_\al')}<\delta(q_\al')
\end{align*}
by \cref{eq:t} since $j^2g(q_\al')\in O_\al$.
Thus, it follows that for any $\al\in \N$,
\begin{align}\label{eq:al_1}
    \al\geq \al_1 \Longrightarrow  \norm{(\psi\circ g)(q_\al')-(\psi\circ f)(q_\al')}<\frac{\ep}{3}.
\end{align}

For each $\al\in \N$, since $q_\al, q_\al'\in U_\al'$, we have 
\begin{align*}
     \norm{(\psi\circ f)(q_\al')-(\psi\circ f)(q_\al)}=
     \norm{\eta_\al(\varphi_\al(q_\al'))-\gamma(\al)}
     = \frac{\norm{\varphi_\al(q_\al')}^2}{2}
     < \frac{r_\al^2}{2}.
\end{align*}
Since $\D\lim_{\al\to \infty}r_\al=0$, there exists $\al_2\in \N$ such that for any $\al\in \N$,
\begin{align}\label{eq:al_2}
    \al\geq \al_2 \Longrightarrow  \norm{(\psi\circ f)(q_\al')-(\psi\circ f)(q_\al)}<\frac{\ep}{3}.
\end{align}

Since $(\psi \circ f)(q_\al)=\gamma(\al)$ for each $\al\in \N$, we have 
\begin{align*}
    \set{(\psi \circ f)(q_\al)|\al\in\N}=\Q^p.
\end{align*}
Hence, there exists $\al_3\in \N$ such that $\al_3>\max\set{\al_1,\al_2}$ and 
\begin{align}\label{eq:al_3}
    \norm{(\psi\circ f)(q_{\al_3})-y_0}<\frac{\ep}{3}.
\end{align}
Thus, we get $\norm{(\psi\circ g)(q_{\al_3}')-y_0}<\ep$ by \eqref{eq:al_0} to \cref{eq:al_3}, which implies that $g(q_{\al_3}')\in V'$.
\end{proof}
\smallskip 
\underline{STEP~3}.
The purpose of this step is to show that $\mathscr{U}_d$ is dense in $\mathscr{U}$, where $d\geq 2$ is a given integer.
Let $g\in \mathscr{U}$ be an arbitrary mapping, and let $\mathscr{U}_g$ be any open neighborhood of $g$.
Then, there exist a non-negative integer $k$ and an open set $O'$ of $J^k(N,P)$ such that 
\begin{align*}
    g\in \set{h\in C^\infty(N,P)|j^kh(N)\subset O'}\subset \mathscr{U}_g.
\end{align*}
For the proof, it is sufficient to show that there exists a mapping $h\in \mathscr{U}_d$ such that $j^kh(N)\subset O'$.

For any $\al\in \N$ and $c\in \R^p$, let $G_{\al,c}:N\to P$ be the mapping defined by 
\begin{align*}
    G_{\al,c}=\psi^{-1}\circ (\psi \circ g+\rho_\al c).
\end{align*}
\begin{lemma}[{\cite[Lemma~3.2]{Ichiki2022}}]\label{thm:contain}
Let $\al$ be any positive integer.
Then, there exists a positive real number $r_\al'$ such that $j^kG_{\al,c}(N)\subset O'$ for any $c\in B^p(0,r_\al')$.
\end{lemma}
Since $g\in \mathscr{U}$, note that there exists a sequence $\set{q_\al'}_{\al\in \N}$ of points in $N$ satisfying \cref{thm:u_critical} and \cref{thm:u_dense} of \cref{thm:u}.
The following lemma can be also shown by the same method as \cite[Lemma~3.3]{Ichiki2022}.
\begin{lemma}\label{thm:summary}
Let $m$ be any positive integer.
Then, there exist $(m+1)$ distinct positive integers $\al_1\ld \al_{m+1}$ and $m$ positive real numbers $r_{\al_1}'\ld r_{\al_m}'$ $(r_{\al_1}'>\cdots >r_{\al_m}')$ such that for any $j\in \set{1\ld m}$,
\begin{enumerate}[$(1)$]
    \item 
    $j^kG_{\al_j,c}(N)\subset O'$ for any $c\in B^p(0,r_{\al_j}')$,
    \item 
    $\norm{(\psi\circ g)(q_{\al_{j+1}}')-(\psi\circ g)(q_{\al_{j}}')}<\D\frac{r_{\al_{j}}'}{d-1}$, where $d\geq 2$ is given in \cref{thm:main_a}.  
\end{enumerate}
\end{lemma}
\begin{proof}[Proof of \cref{thm:summary}]
We will prove this lemma by induction on $m$.

Let $\al_1$ be a positive integer.
By \cref{thm:contain}, there exists a positive real number $r_{\al_1}'$ such that $j^kG_{\al_1,c}(N)\subset O'$ for any $c\in B^p(0,r_{\al_1}')$.
From \cref{thm:u}~\cref{thm:u_dense}, there exists $\al_2 \in \N\setminus\set{\al_1}$ satisfying
\begin{align*}
    \norm{(\psi\circ g)(q_{\al_2}')-(\psi\circ g)(q_{\al_1}')}<\frac{r_{\al_1}'}{d-1}.
\end{align*}
Thus, the case $m=1$ holds.

We assume that the lemma holds for $m=i$, where $i$ is a positive integer.
By \cref{thm:contain}, there exists a positive real number $r_{\al_{i+1}}'$ $(r_{\al_i}'>r_{\al_{i+1}}')$ such that $j^kG_{\al_{i+1},c}(N)\subset O'$ for any $c\in B^p(0,r_{\al_{i+1}}')$.
From \cref{thm:u}~\cref{thm:u_dense}, there exists $\al_{i+2}\in \N\setminus\set{\al_1\ld \al_{i+1}}$ satisfying $\norm{(\psi\circ g)(q_{\al_{i+2}}')-(\psi\circ g)(q_{\al_{i+1}}')}<\frac{r_{\al_{i+1}}'}{d-1}$.
Therefore, the case $m=i+1$ holds.
\end{proof}
For simplicity, set $I=\set{1\ld d-1}$.
By \cref{thm:summary} in the case $m=d-1$, there exist $d$ distinct positive integers $\al_1\ld \al_{d}$ and $d-1$ positive real numbers $r_{\al_1}'\ld r_{\al_{d-1}}'$ $(r_{\al_1}'>\cdots >r_{\al_{d-1}}')$ such that for any $j\in I$,
\begin{enumerate}[(a)]
\setcounter{enumi}{3}
    \item \label{thm:summary_contain}
    $j^kG_{\al_j,c}(N)\subset O'$ for any $c\in B^p(0,r_{\al_j}')$,
    \item \label{thm:summary_i}
    $\norm{(\psi\circ g)(q_{\al_{j+1}}')-(\psi\circ g)(q_{\al_{j}}')}<\D\frac{r_{\al_{j}}'}{d-1}$.
\end{enumerate}
Let $h:N\to P$ be the mapping defined by 
\begin{align*}
    h=\psi^{-1}\circ \left(\psi\circ g+\sum_{i=1}^{d-1}\rho_{\al_i}c_i\right),
\end{align*}
where $c_i=(\psi\circ g)(q_{\al_{d}}')-(\psi\circ g)(q_{\al_i}')\in \R^p$.

First, we show that $j^kh(N)\subset O'$.
Let $q\in N$ be an arbitrary point.
In the case where $q$ is an element of $(\bigcup_{j=1}^{d-1}\supp \rho_{\al_j})^c$, since $h=g$ on the open neighborhood $(\bigcup_{j=1}^{d-1}\supp \rho_{\al_j})^c$ of $q$, we have $j^kh(q)=j^kg(q)\in O'$.
We consider the case where there exists $j\in I$ satisfying  $q\in \supp \rho_{\al_j}$.
Since $\supp \rho_{\al_j}\subset \bigcap_{i\in I\setminus \set{j}}(\supp \rho_{\al_i})^c$ and $h=G_{\al_j,c_j}$ on the open neighborhood  $\bigcap_{i\in I\setminus \set{j}}(\supp \rho_{\al_i})^c$ of $q$, we obtain $j^kh(q)=j^kG_{\al_j,c_j}(q)$.
Moreover, since 
\begin{equation}
\begin{split}\label{eq:c}
    \norm{c_j}&=\norm{(\psi\circ g)(q_{\al_{d}}')-(\psi\circ g)(q_{\al_{j}}')}
    \\
    &\leq\sum_{i=j}^{d-1}\norm{(\psi\circ g)(q_{\al_{i+1}}')-(\psi\circ g)(q_{\al_{i}}')}
    \\
    &<\sum_{i=j}^{d-1}\frac{r_{\al_{i}}'}{d-1}
    \\
    &\leq r_{\al_{j}}',
\end{split}
\end{equation}
we have $c_j\in B^p(0,r_{\al_j}')$.
The last two inequalities in \cref{eq:c} follow from \cref{thm:summary_i} and $r_{\al_j}'>\cdots >r_{\al_{d-1}}'$, respectively.
Thus, we obtain $j^kG_{\al_j,c_j}(q)\in O'$ by \cref{thm:summary_contain}, which implies that $j^kh(q)\in O'$.

Now, we show that $h\in \mathscr{U}_d$.
For any $i,j\in I$, since $\rho_{\al_i}(q_{\al_j}')=\delta_{ij}$ and $\rho_{\al_i}(q_{\al_{d}}')=0$, we obtain{\small
\begin{align*}
    \left(\psi\circ g+\sum_{i=1}^{d-1}\rho_{\al_i}c_i\right)(q_{\al_j}')
    =(\psi\circ g)(q_{\al_j}')+c_j
    =(\psi\circ g)(q_{\al_{d}}')
    =(\psi\circ h)(q_{\al_{d}}'),
\end{align*}}where $\delta_{ij}$ is the Kronecker delta.
Thus, we have $h(q_{\al_1}')=\cdots= h(q_{\al_{d}}')$.
Moreover, for any $j\in I$, the point $q_{\al_j}'$  (resp. $q_{\al_{d}}'$) is a topologically critical point of $h$ since $h=\psi^{-1}\circ (\psi\circ g+c_j)$ on an open neighborhood of $q_{\al_j}'$ (resp. $h=g$ on an open neighborhood of $q_{\al_{d}}'$).
Therefore, we obtain $h\in \mathscr{U}_d$.

Finally, we will show that any mapping in $\mathscr{U}$ is not $C^0$-stable.
Suppose that there exists a $C^0$-stable mapping $\Tilde{g}$ in $\mathscr{U}$.
Then, there exists an open neighborhood $\mathscr{U}_{\Tilde{g}}$ of $\tilde{g}$ such that any mapping in $\mathscr{U}_{\Tilde{g}}$ is $C^0$-equivalent to $\Tilde{g}$.
Since we have shown that $\mathscr{U}_{p+1}$ is dense in $\mathscr{U}$, and topologically critical points are preserved by homeomorphisms as mentioned in \cref{rem:top}\cref{rem:top_p},  any mapping in $\mathscr{U}_{\Tilde{g}}$ has $(p+1)$ topologically critical points with the same image.
This contradicts the fact that the set of all mappings with normal crossings is dense in $C^\infty(N,P)$.
\QED

\section*{Acknowledgements}
The author would like to thank Toru Ohmoto for his kind comments.
This work was supported by JSPS KAKENHI Grant Number JP21K13786.


\begin{thebibliography}{10}

\bibitem{Dimca1979}
Alexandru Dimca.
\newblock Morse functions and stable mappings.
\newblock {\em Rev. Roumaine Math. Pures Appl.}, 24(9):1293--1297, 1979.

\bibitem{Wall1995}
Andrew du~Plessis and Terry Wall.
\newblock {\em The geometry of topological stability}, volume~9 of {\em London
  Mathematical Society Monographs. New Series}.
\newblock Oxford Science Publications. The Clarendon Press, Oxford University
  Press, New York, 1995.

\bibitem{Gibson1976}
Christopher~G. Gibson, Klaus Wirthmüller, Andrew~A. du~Plessis, and Eduard
  J.~N. Looijenga.
\newblock {\em Topological stability of smooth mappings}, volume 552 of {\em
  Lecture Notes in Mathematics}.
\newblock Springer-Verlag, Berlin-New York, 1976.

\bibitem{Ichiki2022}
Shunsuke Ichiki.
\newblock Non-density of stable mappings on non-compact manifolds.
\newblock {\em Pure Appl. Math. Q.}, 19(2):515--527, 2023.

\bibitem{Mather1968}
John~Norman Mather.
\newblock {Stability of {$C^\infty$} mappings. I. The division theorem}.
\newblock {\em Ann. of Math. (2)}, 87:89--104, 1968.

\bibitem{Mather1968b}
John~Norman Mather.
\newblock {Stability of $C^\infty$ mappings. III. Finitely determined
  mapgerms}.
\newblock {\em Inst. Hautes \'{E}tudes Sci. Publ. Math.}, 35:279--308, 1968.

\bibitem{Mather1969}
John~Norman Mather.
\newblock {Stability of $C^\infty$ mappings. II. Infinitesimal stability
  implies stability}.
\newblock {\em Ann. of Math. (2)}, 89:254--291, 1969.

\bibitem{Mather1969b}
John~Norman Mather.
\newblock {Stability of $C^\infty$ mappings. IV. Classification of stable germs
  by $R$-algebras}.
\newblock {\em Inst. Hautes \'{E}tudes Sci. Publ. Math.}, 37:223--248, 1969.

\bibitem{Mather1970}
John~Norman Mather.
\newblock {Stability of $C^\infty$ mappings. V. Transversality}.
\newblock {\em Advances in Math.}, 4:301--336, 1970.

\bibitem{Mather1971}
John~Norman Mather.
\newblock {Stability of $C^\infty$ mappings. VI. The nice dimensions}.
\newblock {\em Proceedings of Liverpool Singularities-Symposium, {I}. Lecture
  Notes in Math.}, 192:207--253, 1971.

\bibitem{Mather1973a}
John~Norman Mather.
\newblock Stratifications and mappings.
\newblock {\em Dynamical systems $($Proc. Sympos., Univ. Bahia, Salvador,
  1971$)$}, pages 195--232, 1973.

\bibitem{Ruas2022}
Maria Aparecida~Soares Ruas.
\newblock {\em Old and new results on density of stable mappings}, volume 552
  of {\em Handbook of geometry and topology of singularities III}.
\newblock Springer, Cham, 2022.

\end{thebibliography}
\end{document}